\documentclass[english]{amsart}
\usepackage[T1]{fontenc}
\usepackage[utf8]{inputenc} 
\usepackage{amsmath}
\usepackage{amsthm}
\usepackage{amssymb}
\usepackage{mathrsfs}
\usepackage{mathtools}

\usepackage[bookmarks]{hyperref}
\usepackage{enumitem}

\usepackage{tipa}

\theoremstyle{plain}
\newtheorem{thm}{\protect\theoremname}[section]
  \theoremstyle{plain}
  \newtheorem{fact}[thm]{\protect\factname}
  \theoremstyle{definition}
  \newtheorem{defn}[thm]{\protect\definitionname}
  \newtheorem{nota}[thm]{\protect\notationname}
  \theoremstyle{plain}
  \newtheorem{lem}[thm]{\protect\lemmaname}
  \theoremstyle{plain}
  \newtheorem{prop}[thm]{\protect\propositionname}

\usepackage{babel}
  \providecommand{\definitionname}{Definition}
  \providecommand{\notationname}{Notation}
  \providecommand{\factname}{Fact}
  \providecommand{\lemmaname}{Lemma}
  \providecommand{\propositionname}{Proposition}
\providecommand{\theoremname}{Theorem}

 \newtheoremstyle{named}{}{}{\itshape}{}{\bfseries}{.}{.5em}{\thmnote{#3}}
 \theoremstyle{named}
 \newtheorem*{namedtheorem}{Theorem}

\providecommand{\dotdiv}{
  \mathbin{
    \vphantom{+}
    \text{
      \mathsurround=0pt 
      \ooalign{
        \noalign{\kern-.35ex}
        \hidewidth$\smash{\cdot}$\hidewidth\cr 
        \noalign{\kern.35ex}
        $-$\cr 
      }%
    }%
  }%
}

\newcommand{\Lc}{\mathcal{L}}
\newcommand{\sforall}{\ensuremath\forall\mkern-7.2mu\forall}
\newcommand{\wexists}{\text{\Large\textschwa}}

\newcommand{\frk}{\mathfrak}

\newcommand{\D}{\mathsf{D}^{\geq 1}}
\newcommand{\DS}{\mathsf{D}^{= 1}}

\begin{document}


\title{Metric Spaces Are Universal for Bi-interpretation with Metric Structures}

\author{James Hanson}
\email{jehanson2@wisc.edu}
\address{Department of Mathematics, University of Wisconsin--Madison, 480 Lincoln Dr., Madison, WI 53706}
\date{\today}


\begin{abstract}
  In the context of metric structures introduced by Ben Yaacov, Berenstein, Henson, and Usvyatsov
  \cite{MTFMS}, we exhibit an explicit encoding of metric structures in countable signatures as pure metric spaces in the empty signature, showing that such structures are universal for bi-interpretation among metric structures with positive diameter. This is analogous to the classical encoding of arbitrary discrete structures in finite signatures as graphs, but is stronger in certain ways and weaker in others. There are also certain fine grained topological concerns with no analog in the discrete setting.
\end{abstract}

\maketitle


\section{Introduction}
\noindent It is a well known fact \cite{hodges_1993}  
 that any discrete structure with finite signature can be encoded as a graph in a particularly strong way:
\begin{fact} \label{fact:main-fact}
For any finite signature $\mathcal{L}$ there is a sentence $\chi$ in a language with a single binary predicate such that every model of $\chi$ is a graph and the class of models of $\chi$ is bi-interpretable with the class of $\mathcal{L}$-structures with more than one element. Furthermore this bi-interpretation preserves embeddings and is computable in the sense that presentations of models of $\chi$ are uniformly computable from presentations of the corresponding $\mathcal{L}$-structure and vice versa.
\end{fact}

This immediately implies that the set of tautologies involving a single binary predicate is undecidable, even though monadic first-order logic, which involves only unary predicates, is decidable. This is in contrast to the situation in continuous first-order logic, introduced in \cite{MTFMS}. There is an easy encoding of a graph $(V,E)$ as a metric space $(V,d)$
wherein 
\[
d(x,y)=\left\{ \begin{array}{cc}
0 & x=y\\
\frac{1}{2} & xEy\\
1 & \text{otherwise}
\end{array}\right..
\]
So the set of continuous tautologies in the empty signature is undecidable for any reasonable notion of computable continuous formulas. Moreover, discrete structures can be encoded as metric spaces, in light of Fact \ref{fact:main-fact}.

The proof of Fact \ref{fact:main-fact} uses a `tag construction,' in which each tuple $x_0,x_1,\allowbreak ...,\allowbreak x_{k-1}$ related by some relation $P$ is connected by a tag which is engineered to distinguish each $x_i$ and to be distinguishable from tags corresponding to relations other than~$P$.

The `tag construction' does not generalize in any satisfactory way to metric structures. Nevertheless, we are able to prove a generalization of Fact~\ref{fact:main-fact}---namely our main result, Theorem~\ref{thm:main2}---using a more intricate construction. 
The full statement of Theorem~\ref{thm:main2} is somewhat technical, but we can summarize the important part in the following.

\begin{namedtheorem}[Summary of Theorem~\ref{thm:main2}] 
For any countable metric signature $\mathcal{L}$ and $r>0$, there is a theory $T$ in the empty signature such that the class of models of $T$ is bi-interpretable with the class of $\mathcal{L}$-structures with diameter $\geq r$. This bi-interpretation preserves embeddings and $d$-finiteness of types. If the original structure is not strongly infinite dimensional, then the interpreted structure will also not be strongly infinite dimensional. Furthermore, the bi-interpretation is computable in the sense that presentations of models of $T$ are uniformly computable from presentations of the corresponding $\mathcal{L}$-structure and vice versa.
\end{namedtheorem}

 There are some improvements in Theorem \ref{thm:main2} over Fact \ref{fact:main-fact}, namely that the encoding works in the empty signature\textemdash which is largely cosmetic\textemdash and that we can encode countable signatures rather than just finite ones.
\emph{$d$-finiteness of types} is a technical niceness condition introduced in \cite{Yaacov2007} that will be discussed below. Strong infinite dimensionality is relevant from the point of view of computable structure theory, as the continuous degree of a point in a finite dimensional or weakly infinite dimensional metric space is always total \cite{2014arXiv1405.6866K}. These two concepts play no essential role in the construction, although they do motivate a particular choice in it, namely using a disjoint union construction rather than a product construction.

The restriction that the metric structures have diameter uniformly bounded below is the necessary analog of the `more than one element' restriction. A simple compactness argument shows that we could never have uniform bi-interpretability between a single elementary class of metric spaces and the class of all $\mathcal{L}$-structures of positive diameter. In both the discrete case and the metric case we could avoid this non-uniformity by appending a new sort to every structure that always contains precisely two elements distance 1 apart. Also it should be noted that this is a non-issue from a computable structure theory point of view, since a one point structure is clearly computable.

Finally there is the issue of finite axiomatizability, which the generalization loses. As will be discussed at the end of the paper in Section \ref{sec:fin}, however, there is no clear analog of finite axiomatizability in continuous logic.




\section{Preliminaries}

\noindent In the interest of notational brevity, we will describe one step of
the bi-interpretation informally before defining the concept of a
metric signature rigorously: 
\begin{fact}
Every many-sorted metric signature can be recast as a purely relational
metric signature with $[0,1]$-valued predicates and metrics.
\end{fact}

From now on all predicate symbols will be $[0,1]$-valued and in particular all sorts will have diameter $\leq 1$.

There are some trivial subtleties if we allow ourselves predicates
with zero-length ranges or other such bookkeeping edge  
cases, but I trust that anyone dedicated enough to include those in
their formalism will be more than capable of resolving those issues
on their own. Normally the task of setting out the bookkeeping for
many-sorted structures is similarly relegated, but for our purposes
here it will be prudent to consider it immediately.

For computable metric signatures, obviously we should require that
the predicate ranges and maximum sort diameters be uniformly computable
before recasting in the form above (although really all we need are
uniformly computable upper and lower bounds), in order to ensure that
we can uniformly compute presentations of recast structures from
presentations of the original structures.
\begin{defn}
  \begin{enumerate}[label=(\roman*)]
  \item  A \emph{metric signature} $\mathcal{L}$, is a tuple $(\mathcal{S},\mathcal{P},a,\Delta)$,
where
\begin{itemize}
\item $\mathcal{S}$ is a set of sort symbols;
\item $\mathcal{P}$ is a set of predicate symbols;
\item $a:\mathcal{P}\rightarrow\mathcal{S}^{<\omega}$ is the arity function
that assigns to each predicate symbol its finite string of input sorts
(by an abuse of notation, we will use $a$ for formulas as well as
atomic predicates); and
\item for each predicate symbol $P$, $\Delta_{P}:[0,1]\rightarrow[0,1]$
is the syntactic modulus of uniform continuity of $P$.
\end{itemize}
\item A \emph{computable metric signature} is a metric signature such that
$\mathcal{S}$ and $\mathcal{P}$ are computable subsets of $\omega$,
$a$ is a computable function which is total on $\mathcal{P}$, and
$P\mapsto\Delta_{P}$ is a uniformly computable family of total computable
functions.
  \end{enumerate}
\end{defn}

Although in full generality moduli of uniform continuity can be specified
as functions of each variable individually, on the level of a metric
signature not much is gained by such a generalization. Likewise there
is no particular reason for moduli of uniform continuity to be continuous
anywhere other than $0$, but again very little is gained and continuity
is a more natural convention in the context of computable metric signatures.

The phrase `syntactic modulus of uniform continuity' refers to the
fact that in a given $\mathcal{L}$-structure the corresponding predicate
may obey a stricter modulus of uniform continuity.

The definitions of restricted $\mathcal{L}$-formulas, $\mathcal{L}$-structures,
and other such things is given in \cite{MTFMS}.

We should be clear about what a computable metric structure is.
\begin{defn}
Given a computable metric signature $\mathcal{L}$, a \emph{computable $\mathcal{L}$-structure}
is an $\mathcal{L}$-structure whose universes are a uniformly computable
family of computable metric spaces (in the sense of \cite{Weihrauch:2000:CAI:358357}) and whose
predicate interpretations are all uniformly computable functions.
\end{defn}

Finally we will need a syntactically uniform notion of definable predicate,
similar to the one given in \cite{10.2307/25747120}.
\begin{defn}

  \begin{enumerate}[label=(\roman*)]
  \item For a metric signature $\mathcal{L}$, a \emph{finitary $\mathcal{L}$-formula}
is an expression of the form $\sum_{n<\omega}2^{-(n+1)}\varphi_{n}$,
with $\varphi_{n}$ a sequence of $[0,1]$-valued restricted $\mathcal{L}$-formulas
such that the entire sequence contains finitely many free variables. Such a formula has a syntactic
modulus of uniform continuity of $\sum_{n<\omega}2^{-(n+1)}\Delta_{\varphi_{n}}$.
\item  An \emph{$\omega$-infinitary $\mathcal{L}$-formula} is an expression
of the same form allowing possibly infinitely many free variables.\footnote{Such an expression has a uniformly computable syntactic
modulus of uniform continuity in terms of the appropriate metric on
$\omega$-tuples, but it is somewhat more complicated to state.}
\item  An \emph{$\mathcal{L}$-formula} is either a finitary or an $\omega$-infinitary
$\mathcal{L}$-formula.
\item  A \emph{computable $\mathcal{L}$-formula} is an $\mathcal{L}$-formula such that the sequence of formulas $\varphi_n$ is computable (the $\varphi_n$ are required to be restricted formulas and can therefore be encoded by natural numbers).
  \end{enumerate}

\end{defn}

In \cite{MTFMS}, definable predicates are defined relative to a single metric structure or an elementary class of metric structures in terms of uniformly convergent limits of restricted formulas, but the notion of formula given here, which is alluded to as a possibility in \cite{MTFMS}, is purely syntactic and can be interpreted in any $\mathcal{L}$-structure.

With some straightforward work, one can check that any definable
predicate in the typical
sense can be written in this form and that the resulting family of
formulas is closed under this `infinitary connective,' as well as all ordinary connectives, up to logical
equivalence. In particular, even though continuous logic in some sense has an infinitary conjunction, it does not have a tall hierarchy of infinitary formulas the same way that $\mathcal{L}_{\omega_1 \omega}$ does. Furthermore, one can show that this is computably true as well.

  \begin{fact}
    If $\{\varphi_n\}_{n<m}$ is a uniformly computable sequence of $\Lc$-formulas for some $m \leq \omega$ and $F:[0,1]^m \to [0,1]$ is a computable function, then $F(\bar{\varphi})$ is logically equivalent to a computable $\Lc$-formula. Furthermore the equivalent formula is uniformly computable in $\{\varphi_n\}_{n<m}$ and $F$.
  \end{fact}


\subsection{Closed and Open Formulas and Definability Quantifiers}

There are many real valued sentences in this paper which are meant to capture an intuitive notion (such as being a bijection). In the interest of making this intuition clear, we will use a notation that mimics ordinary first-order logic as closely as possible but does not change the meaning of any established logical symbols. There are precedents for this kind of notation in the precursors of continuous logic, and there are many instances of continuous logicians slipping into something similar informally in the literature.

In order to do this without modifying the meaning of any existing logical symbols, we will need two new quantifiers.

\begin{defn}
  A structure $\frk{M}$ satisfies $\sforall x \varphi(x)$, written $\frk{M} \models \sforall x \varphi(x)$ if for every elementary extension $\frk{N} \succeq \frk{M}$ and every $a \in \frk{N}$, $\frk{N} \models \varphi(a)$. This is called \emph{strong universal quantification}.

  A structure $\frk{N}$ satisfies $\wexists x\varphi(x)$, written $\frk{M} \models \wexists x \varphi(x)$ if for some elementary extension $\frk{N} \succeq \frk{M}$ and some $a \in \frk{N}$, $\frk{N} \models \varphi(a)$. This is called \emph{weak existential quantification}.
\end{defn}

Recall that a \emph{condition} is an equality or inequality involving two real valued formulas (often with one of them a constant). Conditions involving $=$, $\leq$, and $\geq$ are \emph{closed} and conditions involving $\neq$, $>$, and $<$ are \emph{open}. These correspond to closed and open subsets of type space.\footnote{Although in general not all closed or open subsets of type space are of this form. Closed conditions correspond precisely to closed $G_\delta$ subsets of type space, and open conditions correspond precisely to open $F_\sigma$ subsets of type space.}

\begin{defn}
  The classes of \emph{closed} and \emph{open formulas} are defined inductively.
  \begin{itemize}
  \item Any closed condition is a closed formula.
  \item Any open condition is an open formula.
  \end{itemize}
  If $F$ and $G$ are closed formulas and $U$ and $V$ are open formulas, then
  \begin{itemize}
  \item $F \wedge G$, $F \vee G$, $U \to F$, and $\neg U$ are closed formulas,
  \item $U \wedge V$, $U \vee V$, $F \to U$, and $\neg F$ are open formulas,
  \item $\forall x F$ and $\wexists x F$ are closed formulas, and
  \item $\sforall x U$ and $\wexists x U$ are open formulas.
  \end{itemize}
\end{defn}

Satisfaction of closed and open formulas is defined in the obvious way, as is the notion of free variables. It is not hard to show that every closed (resp.\ open) formula is logically equivalent to a closed (resp.\ open) condition, and that this is witnessed by an explicit computable mapping. From this it follows that the set of types satisfying a closed (resp.\ open) formula is topologically closed (resp.\ open), justifying the name. We will use lowercase Greek letters for real valued formulas and uppercase Roman letters for closed and open formulas.

\begin{nota}
  We may use $x=y$ as shorthand for the closed formula $d(x,y) = 0$.
\end{nota}

\begin{nota}
  If $\varphi(\bar{x},y)$ is a real valued formula, then we write $\D y\varphi(\bar{x},y)$ for
\begin{align*}
  \forall y \left[ \wexists z (\varphi(\bar{x},z) = 0 \wedge d(y,z)=\varphi(\bar{x},y)) 
               \wedge \forall z( \varphi(\bar{x},y) \leq \varphi(\bar{x},z) + d(y,z))\right]. 
\end{align*}  
\end{nota}
This is a re-expression of the axioms given in \cite[Theorem 9.12]{MTFMS} that capture that $\varphi$ is the distance predicate of a non-empty definable set. 
In other words, $\frk{M} \models \D y\varphi(\bar{a},y)$ if and only if $\varphi^{\frk{M}}(\bar{a},y)$ is the distance predicate of a non-empty definable set.

\begin{nota}
  If $\varphi(\bar{x},y)$ is a real valued formula, then we write $\DS y \varphi(\bar{x},y)$ for
  \[
    \wexists y (\varphi(\bar{x},y)=0\wedge \forall z(d(y,z) = \varphi(\bar{x},z))).
  \]
\end{nota}

It is not hard to show that $\frk{M} \models \DS y \varphi(\bar{a},y)$ if and only if $\varphi^{\frk{M}}(\bar{a},y)$ is the distance predicate of a singleton. Therefore $\varphi(\bar{x},y)$ defines a function in $\frk{M}$ if and only if $\frk{M} \models \forall \bar{x} \DS y \varphi(\bar{x},y)$.

Note that $\D x$ and $\DS y$ are, syntactically speaking, quantifiers which take real-valued formulas and produce closed formulas.

We will frequently use the following fact, which was originally shown in \cite[Theorem 9.12]{MTFMS}.\footnote{Strictly speaking they only show this for formulas of the form $\varphi(x)$, but the extension to uniformly definable families satisfying $\forall\bar{y} \D x \varphi(x,\bar{y})$ is immediate.}

\begin{fact}\label{fact:rel-quan}
  If $T \models \forall\bar{y} \D x \varphi(x,\bar{y})$, then
  \begin{itemize}
  \item for any real-valued formula $\psi(x,\bar{y},\bar{z})$, there is a real-valued formula $\eta(\bar{y},\bar{z})$ such that for any $\frk{M} \models T$ and any $\bar{b}\bar{c} \in \frk{M}$, $\eta^{\frk{M}}(\bar{b},\bar{c}) = \inf\{ \psi(a,\bar{b},\bar{c}) : a \in \frk{M}\text{, }\varphi^{\frk{M}}(a,\bar{b})\}$,
  \item the same with $\sup$ instead of $\inf$,
  \item for any closed formula $F(x,\bar{y},\bar{z})$, there is a closed formula $G(\bar{y},\bar{z})$ such that for any $\frk{M} \models T$ and any $\bar{b}\bar{c} \in \frk{M}$, $\frk{M} \models G(\bar{b},\bar{c})$ if and only if $\frk{M} \models F(a,\bar{b},\bar{c})$ for every $a \in \frk{M}$ with $\varphi^{\frk{M}}(a,\bar{b}) = 0$,
  \item for any closed formula $F(x,\bar{y},\bar{z})$, there is a closed formula $H(\bar{y},\bar{z})$ such that for any $\frk{M} \models T$ and any $\bar{b} \bar{c} \in \frk{M}$, $\frk{M} \models H(\bar{b},\bar{c})$ if and only if there is an elementary extension $\frk{N} \succeq \frk{M}$ and an $a \in \frk{N}$ such that $\frk{N} \models F(a,\bar{b},\bar{c})$, and
  \item the analogous statements for open formulas.
  \end{itemize}

  If $T \models \forall \bar{y} \DS x \varphi(x,\bar{y})$, then for every model $\frk{M} \models T$, there is a function $f : \frk{M}^{|\bar{y}|} \to \frk{r}$ such that for any $\bar{b} \in \frk{M}$, $\varphi(x,\bar{b})$ is the distance predicate of the singleton $\{f(\bar{b})\}$. For any (real-valued, closed, or open) formula $X(\bar{y},\bar{z},w)$, there is a formula $Y(\bar{y},\bar{z})$ logically equivalent to $X(\bar{y},\bar{z},f(\bar{y}))$ in every model of $T$.
  
  Furthermore, these formulas can be produced in a uniformly computable way.
\end{fact}
  
In light of these facts, we will use standard notation for relative quantification (i.e.\ expressions such as $\sup_{x \in D}$, $(\forall x \in D)$, and $(\wexists x \in D)$), and we will use common notation for definable functions and constants.

\section{Expansions}

\noindent We need to specify a few notions of expansions and interdefinability
in continuous logic.


\begin{defn}
  \begin{enumerate}[label=(\roman*)]
  \item For a given metric signature $\mathcal{L}$ and a finitary $\mathcal{L}$-formula
$\varphi(\overline{x})$, a \emph{definitional expansion} of $\mathcal{L}$
by $\varphi$ is a metric signature $\mathcal{L}^{\ast}$ containing
the same sorts as $\mathcal{L}$ and a single new predicate symbol
$P$ with $a(\varphi)=a(P)$ and $\Delta_{\varphi}=\Delta_{P}$. For
an $\mathcal{L}$-structure $\mathfrak{A}$, the corresponding $\mathcal{L}^{\ast}$-structure
$\mathfrak{A}^{\ast}$ is given by interpreting $P$ as $\varphi$.
We also refer to iterated definitional expansions as definitional
expansions.
\item An $\mathcal{L}$-structure $\mathfrak{A}$ and a $\mathcal{K}$-structure
$\mathfrak{B}$ are \emph{interdefinable} if there are definitional expansions
$\mathfrak{A}^{\ast}$ and $\mathfrak{B}^{\ast}$ which make them
isomorphic up to relabeling of sorts and predicate symbols. (We allow
metrics to be relabeled.) 

An elementary class $C_{0}$ of $\mathcal{L}$-structures and an elementary class $C_{1}$
of $\mathcal{K}$-structures are \emph{interdefinable} if there are functors $F: C_{0} \to C_{1}$ and $G: C_{1} \to C_{0}$ given by uniform definitional expansions and relabelings which form an equivalence of categories, where we treat $C_0$ and $C_1$ as categories with elementary embeddings as morphisms (i.e.\ $F\circ G$ and $G\circ F$ are both naturally isomorphic to the identity functor\footnote{That is to say, for each structure $\frk{A}$ in $C_0$, there is a designated isomorphism $\alpha_{\frk{A}}: \frk{A} \to G\circ F (\frk{A})$ such that for any $\frk{B} \in C_0$ and any elementary map $f : \frk{A} \preceq \frk{B}$, $\alpha_{\frk{B}} \circ f = (G \circ F(f))\circ \alpha_{\frk{A}}$. And likewise for $C_1$.}). 

Note
that we aren't requiring that the syntactic moduli of continuity match. 

\item Given a metric structure $\mathfrak{A}$, an \emph{imaginary expansion}
of $\mathfrak{A}$ is one of the following operations:
\begin{enumerate}[label=$\bullet$,leftmargin=1em]
\item Appending a product sort $P=\prod_{i<n}O_{i}$ for some $0<n\leq\omega$
and sorts $O_{i}\in\mathcal{S}$. By convention the metric on a finitary
product sort will always be the maximum of the component metrics and
the metric on an $\omega$-product sort will always be $\sup_{i<\omega}2^{-i}d_{O_{i}}$. We also append projection predicates $\pi_i$ on $P \times O_i$ for each $i<n$, where $\pi_i(\left<x_0, x_1, \dots, x_{n-1} \right>,y)=d_{O_i}(x_i,y)$.
\item Appending a $\varnothing$-definable set $D$ in sort $O$ as a new sort $O_D$ together with an inclusion predicate $\iota$ on $O_D \times O$, where $\iota(x,y)=d_O(z,y)$ for $x\in O_D$ and $y,z\in O$ with $z$ the element of $D$ corresponding to $x$.
 The metric $d_{O_D}$ is the restriction of $d_O$ to $D$.
\item For $\rho$, a $\varnothing$-definable pseudo-metric on sort $O$,
appending the quotient sort $O/\rho$ along with a quotient predicate $q$ on $O \times O/\rho$, where $q(x,\left[ y \right]_\rho) = \rho(x,y)$ for $x,y\in O$, where $\left[ y \right]_\rho$ is the $\rho$-equivalence class of $y$. This is well-defined because $\rho$ is a pseudo-metric.
\end{enumerate}
  \end{enumerate}
We also refer to iterated imaginary
expansions as imaginary expansions.\footnote{Even though imaginary
expansions are defined for structures and not signatures, of the three forms of imaginary expansion, only expansion by a definable set is not uniform across all structures of a given signature, as every definable pseudo-metric can be written in the form $\rho(x,y)=\sup_ {\overline{z}}|\varphi(x,\overline{z})-\varphi(y,\overline{z})| $, since $\rho(x,y) = \sup_z |\rho(x,z)-\rho(y,z)|$ always holds, and such an expression is a definable pseudo-metric in any $\mathcal{L}$-structure.}

\end{defn}

Recall that we have restricted ourselves to relational languages at this point, which is why the projection, inclusion, and quotient maps are encoded as predicates.

The added generality of allowing $\omega$-tuples and passing to definable
sets is natural and somewhat necessary in continuous logic \cite[Section 11]{MTFMS}. $\omega$-tuples
are necessary for canonical parameters since a formula can involve
countably many parameters. Note that for any formula $\varphi$
on an $\omega$-product sort, if $\mathfrak{A}\models \varphi(\overline{a})$, then for any $\varepsilon>0$ the fact that $\mathfrak{A}\models\varphi(\overline{a})<\varepsilon$ only depends on finitely many terms in $\overline{a}$, uniformly
as a function of $\varepsilon$, because $\varphi$ needs to be uniformly
continuous with regards to the $\omega$-product metric. Because of this, $\omega$-product
sorts are just as safe as finitary product sorts in terms of compatibility
with ultraproducts and preserving the category of models. Explicitly
passing to definable sets is necessary in situations such as the following:
In a connected metric structure $\mathfrak{M}$ with a non-trivial
definable discrete subset $D$, there is no uniformly continuous pseudo-metric
$\rho$ on $\mathfrak{M}$ that will make $\mathfrak{M}/\rho$ isometric
to $D$ (or $D$ plus a single new point or anything else you would do in the discrete setting), since the quotient map $\mathfrak{M}\rightarrow\mathfrak{M}/\rho$
is continuous and continuous functions preserve connectedness. 


\begin{lem}\phantomsection\label{lem:prod}
  \begin{enumerate}[label=(\roman*)]
  \item For any metric signature $\mathcal{L}$ (not necessarily countable), there is a metric signature
$\mathcal{K}$ which is interdefinable with an imaginary expansion of $\mathcal{L}$ such
that $\mathcal{K}$ has a uniform bound of $2$ on the arities of
its predicate symbols. For computable signatures, the signature $\mathcal{K}$
is uniformly computable from $\mathcal{L}$, and presentations of $\mathcal{L}$-structures
can be uniformly converted into corresponding presentations of $\mathcal{K}$-structure and vice versa.
\item There is a $\mathcal{K}$-theory $T_\mathcal{L}$, uniformly computable from $\mathcal{L}$, such that the models of $T_\mathcal{L}$ are precisely the interpretations of $\mathcal{L}$-structures.
  \end{enumerate}
\end{lem}

\begin{proof}
(i) For each predicate symbol $p$, we can define a unary formula on the
sort $\Pi_{O\in a(p)}O$ in the obvious way. These, together with
projection maps between product sorts and the original $\mathcal{L}$-sorts,
are clearly enough to define any predicate originally definable in
an $\mathcal{L}$-structure in a completely uniform way. Since
the projection maps are encoded as $2$-ary predicates, we have the
required arity bound. This procedure is also clearly uniformly computable,
both for signatures and presentations of structures.

(ii) All that $T_\mathcal{L}$ needs to say is that the predicates corresponding to projection maps are actually projection maps and that the product sorts are products of the sorts they project onto.
\end{proof}
\begin{defn}
For any metric signature $\mathcal{L}$ with designated home sort $H$ and any real number $r$ satisfying $0<r\le1$, $C_{\mathcal{L},r}$ is the class
of $\mathcal{L}$-structures $\mathfrak{A}$ satisfying $\text{diam}(H^{\mathfrak{A}})\geq r$.
\end{defn}

The following lemma is the source of all non-uniformity relative to
$r$ in the entire construction and is analogous to the fact that a discrete structure with only one element
cannot interpret any structure with more than one element. It could be avoided by appending
a new compact sort isometric to $[0,1]$ with the standard metric\footnote{Or literally any other fixed non-trivial compact metric space, such as one with two points.}
 and
letting that be the designated home sort $H$.

\begin{lem} \label{lem:comp-im}
Let $X$ be a compact metric space. For structures in the class $C_{\mathcal{L},r}$, there is a uniformly
definable imaginary $Y$ such that for any $\mathfrak{A}\in C_{\mathcal{L},r}$,
$Y^{\mathfrak{A}}\cong X$, with each point of $Y^{\mathfrak{A}}$
and every continuous function $(Y^{\mathfrak{A}})^{n}\rightarrow[0,1]$
uniformly $\varnothing$-definable.
\end{lem}

\begin{proof}
Let $x_0,x_1,y_0,y_1$ be variables in $H$ and consider 
the $\mathcal{L}$-formula 
\[
\rho(x_0,x_1,y_0,y_1)=\min\left\{\frac{1}{r}|d(x_{0},x_{1})-d(y_{0},y_{1})|,1\right\}.
\]
This is a pseudo-metric on $H^{2}$. $H^{2}/\rho$ contains more than
one point for any $\mathfrak{A}\in C_{r}$, because of the diameter
requirement. In particular it has a definable subset consisting of
the $\rho$-equivalence classes of pairs satisfying $d(x_{0},x_{1})=0$ and pairs satisfying
$d(x_{0},x_{1})\geq r$, with each of those points being $\varnothing$-definable
by the formulas $\frac{1}{r}d(x_{0},x_{1})$ and $1\dotdiv\frac{1}{r}d(x_{0},x_{1})$,
respectively. Let $D$ denote this definable set. Clearly $D$ is
always isometric to the discrete space with two points, so in particular
$C=D^{\omega}$ is an isometric copy of Cantor space with the standard
metric with every point uniformly $\varnothing$-definable. It is
well known that Cantor space continuously surjects onto any compact
metric space $(X,d^{X})$, so by pulling back
$d^{X}$ to $C^{2}$ we get a continuous pseudo-metric
on $C$ whose quotient is isometric to $X$. Therefore,
since the type space $S_{C^{2}}(T)$ is isomorphic to $C^{2}$ (both
metrically and topologically), the pullback metric is a continuous
function on $S_{C^{2}}(T)$ and is thus a definable pseudo-metric
on $C$. Since each point of $C$ is uniformly definable, this gives
the required uniformly definable imaginary $Y$. 

Finally for an arbitrary continuous function $f:X^{n}\rightarrow[0,1]$,
the pullback on the type space $S_{C^{n}}(T)$ is continuous and
therefore definable. By construction it is compatible with the quotient
map $C\rightarrow Y$ and is therefore a definable predicate on
the imaginary $Y^{\mathfrak{A}}$.
%
\end{proof}


There are some potential subtleties involving uniform computability of formulas defining computable compact imaginaries and computable predicates on them. In the current context we only need Lemma \ref{lem:comp-im} for a small handful of very specific tame compact metric spaces, so we'll deal with computability on a case-by-case basis. 


\begin{lem} \label{lem:one-sort-fin}
For any $C_{\mathcal{L},r}$, with $r>0$, if  $\{O_n^\mathfrak{A}\}_{n<k}$ is a finite collection of sorts of diameter $\leq 1$, then the disjoint union $U=\bigsqcup_{n<k}O_n$ with metric $d(x,y)=1$ for $x\in O_n$ and $y\in O_m$ with $n\neq m$ and $d(x,y)=d_{O_n}(x,y)$ for $x,y\in O_n$ is a uniformly definable imaginary in $C_{\mathcal{L},r}$. Furthermore the formulas defining $U$ are uniformly computable in $\mathcal{L}$, $r$, and the list of sorts.
\end{lem}
\begin{proof}
By Lemma \ref{lem:comp-im}, the discrete space $\Delta_k = \{0,...,k-1\}$, with the metric $\delta(x,y)=1$ if $x\neq y$, is uniformly an imaginary of $C_{\mathcal{L},r}$ (although in particular we don't have to go through Cantor space, and we can realize $ \Delta_k$ as a quotient of some $\Delta_{2^{\ell}}=(\Delta_2)^{\ell}$ in a uniformly computable way). Furthermore we can arrange that each element of $\Delta_k$ is definable.

Define a formula $\rho(\overline{x},\overline{y})$ on $\Delta_k \times \prod_{n<k}O_n$ by

$$ \rho(\overline{x},\overline{y})=\min\left\{\delta(x_0,y_0)+ \min_{n<k}\left( d(x_{n+1},y_{n+1})+ \delta(x_0,n) \right), 1\right\}.$$

Checking definitions gives that $\Delta_k \times \prod_{n<k}O_n/\rho$ is the required imaginary. This formula is also clearly uniformly computable.
\end{proof}

\section{Countable Disjoint Unions of Sorts}

\noindent A common trick in discrete logic is merging a finite collection of
sorts by taking the disjoint union and adding unary predicates selecting
out each sort. This can't be extended to infinitely many predicates
without changing the category of models; by compactness there will
be models with elements not in any given sort. The added flexibility
of continuous logic allows us to do this with countably many sorts
at once without changing the category of models. Specifically we can arrange it so that any sequence of types that ought to limit to an `unsorted' type is shunted into a single unique overflow point. This is very similar to the emboundment method used in \cite{Yaacov2008-ITACFO} to treat unbounded metric structures.


It should be noted that if $\mathcal{L}$ has finitely many sorts
and (possibly infinitely many) predicates with uniformly bounded arity, this section can
be skipped and the construction in Theorem \ref{thm:main3} will work directly. 
\begin{defn}
Let $\{O_{n}\}_{n<\omega}$ be a countable sequence of $\mathcal{L}$-sorts.
For any $\mathcal{L}$-structure $\mathfrak{A}$, the \emph{(countable) metric disjoint
union} of $\{O_{n}\}_{n<\omega}$, written $\bigsqcup_{n<\omega}^{\ast}O_{n}$, is a metric structure with the 
set 
\[
U^{\mathfrak{A}}=\{\ast\}\cup\bigsqcup_{n<\omega}O_{n}^{\mathfrak{A}}
\]
as its universe,  where $\ast$ is a single new point. 

To define the metric on $U^{\mathfrak{A}}$, let $x,y\in O_{n}^{\mathfrak{A}}$ and $z\in O_{m}^{\mathfrak{A}}$,
with $n\neq m$. Then we have
\begin{align*}
d_{U}^{\mathfrak{A}}(x,y)&=2^{-n}d_{O_{n}}^{\mathfrak{A}}(x,y),\\
d_{U}^{\mathfrak{A}}(x,z)&=|2^{-n}-2^{-m}|,\text{ and}\\
d_{U}^{\mathfrak{A}}(x,\ast)&=2^{-n},
\end{align*}
where the other values are determined by symmetry. We will prove in Proposition \ref{prop:cmdu} that this defines a complete metric space.

A predicate on some $O_{n_{1}}^{\mathfrak{A}}\times\dots\times O_{n_{k}}^{\mathfrak{A}}$
is extended to a predicate on $U^{\mathfrak{A}}$ by setting its value
to $1$ (i.e.\ `false') when the input is not part of its domain. 

Finally we add a distance predicate for the set $\{\ast\}$ (recall that we have restricted ourselves to relational languages, so we can't use a constant). 
\end{defn}

\begin{prop}\phantomsection\label{prop:cmdu}
  \begin{enumerate}[label=(\roman*)]
  \item The countable metric disjoint union, $U=\bigsqcup_{n<\omega}^{\ast}O_{n}$, of a sequence $\{O_{n}\}_{n<\omega}$
of $\mathcal{L}$-sorts is well-defined, i.e.\ the metric given in
the definition is actually a complete metric.
\item The predicates interpreted on it are uniformly continuous. If they are Lipschitz in the original signature, they will still be Lipschitz (although possibly with a different Lipschitz constant).
\item For any fixed $\Lc$ and $r$, the countable metric disjoint union is isomorphic to a uniformly definable imaginary
for all $\mathfrak{A}\in C_{\mathcal{L},r}$. The relevant formulas
and the map of presentations $\mathfrak{A}\mapsto U^{\mathfrak{A}}$
are uniformly computable from the sequence $\{O_{n}\}_{n<\omega}$,
the signature $\mathcal{L}$, and  the real number $r$, 
so in particular if those are all
computable, then the relevant formulas and the map of presentations
are computable.
\item Each $O_{n}$ as a subset of $U$ is a definable subset of $U$
and (considering $U$ as an imaginary sort) there is a definable bijection between $O_{n}$ as a sort and
$O_{n}$ as a definable subset of $U$. The relevant formulas are
uniformly definable in $\mathcal{L}$ and computable.
\item For a fixed sequence $\mathscr{S}=\{O_n\}_{n<\omega}$ of $\mathcal{L}$-sorts with $O_0 = H$, the designated home sort, there is a signature $\mathcal{L}_\mathscr{S}$ and a theory $T_\mathscr{S}$, both uniformly computable from $\mathcal{L}$ and $\mathscr{S}$, and a (real-valued) $\mathcal{L}_\mathscr{S}$-sentence $\Xi_\mathscr{S}$, such that the models of $T_\mathscr{S}\cup \{\Xi_\mathscr{S} \geq r\}$ are precisely the same as reducts to the sort  $\bigsqcup^{\ast}_{n<\omega}O_n$ of structures in $C_{\mathcal{L},r}$.
  \end{enumerate}
\end{prop}

\begin{proof}
(i) The expression given for $d$ clearly obeys all metric space axioms besides the triangle inequality. The only unobvious case is the one consisting of two points in some $O_n$ and a third point in some $O_m$ with $n\neq m$ (where we let $O_\omega = \{\ast\}$ with the understanding that $\text{``}2^{-\omega}\text{''}=0$). 
 Let $x,y\in O_n$ and $z\in O_m$ with $n\neq m$. By symmetry there are only $2$ cases to check:
\begin{itemize}
\item
$d(x,y)\leq 2^{-n}$ and $d(x,z)=d(y,z)=|2^{-n}-2^{-m}|\geq 2^{-(\min\{n, m\} + 1)}$, so $d(x,z)+d(y,z)\geq 2^{-\min\{n,m\}} \geq 2^{-n} \geq d(x,y)$, and in this case the triangle inequality is obeyed.
\item
 $d(x,z) = |2^{-n}-2^{-m}|$, so $d(x,z)\leq d(x,y) + |2^{-n}-2^{-m}|=d(x,y)+d(y,z)$.
\end{itemize}
 So the triangle inequality is obeyed in all cases.

To see that the metric space is complete, note that any Cauchy sequence is either eventually contained in some $O_n$ or limits to $\ast$.
 
(ii) If a predicate $P$ on sort $O_{n_{1}}\times\dots\times O_{n_{k}}$
has modulus of uniform continuity $\Delta_{P}(x)$, then the corresponding
predicate on $U$ is uniformly continuous with modulus
of uniform continuity 
$$\Delta_{P}^{\ast}(x)=\min\{ \max\{\Delta_{P}(\min\{2^{N}x,1\}),2^{N+1}x\} ,1\},$$
where $N=\max\{n_{0},\dots, n_{k-1}\}$. Note that if $P$ has Lipschitz constant $L$, then on $U$ it will have Lipschitz constant $2^{N+1} L$, and in particular it will still be Lipschitz.

(iii) By Lemma \ref{lem:comp-im}, the class $C_{\mathcal{L},r}$ has a uniformly
definable imaginary isometric to the metric space $(X,d)$ where $X=\{0\}\cup\{2^{-n}:n<\omega\}$
and $d$ is the standard metric on $\mathbb{R}$. Let $W=X\times\Pi_{n<\omega}O_{n}$
be the infinitary product sort.

(iv) For any $x\in U$, $d(x,O_{n})=|d(x,\ast)-2^{-n}|$. 

Let $Q:X\rightarrow [0,1]$ be the natural inclusion map, which is a definable predicate on $X$ uniformly for all members of $C_{\mathcal{L},r}$. For each $n$, let 
$$R_n(x)=1 \dotdiv 2^{n+1}|Q(x)-2^{-n}|,$$ 
i.e.\ $R_n$ is a predicate on $X$ which takes on the value $1$ at $2^{-n}$ and $0$ everywhere else. Now define a pseudo-metric on $W$ by $$\rho (\overline{x},\overline{y}) = |Q(x_0)-Q(y_0)| + \sum_{n<\omega} 2^{-n}  R_n (x_0) R_n(y_0)  d_{O_n}(x_{n+1},y_{n+1}) .$$
Although in principle this is $[0,2]$-valued, by construction it will only take on values in $[0,1]$. Taking the quotient $W/\rho$ will identify any two elements $\overline{a},\overline{b}\in W$ if and only if $a_0=b_0$ and either $a_0 = 0$ or $a_0=2^{-n}$ and $a_n=b_n$. So by making the identification of elements of the form $(2^{-n},\dots,a_n,\dots)$ with $a_n$ and elements of the form $(0,\dots)$ with $\ast$, we get a bijection between $W/\rho$ and $U$, and by checking definitions we see that $\rho$ induces the correct metric on $U$.

(v)  The signature $\mathcal{L}_\mathscr{S}$ has a single sort and the same predicate symbols as $\mathcal{L}$ with the same total arity along with a single new unary predicate symbol $Q$. For each predicate symbol $P$, the syntactic modulus of continuity is
\[
  \min\{ \max\{\Delta_{P}(\min\{2^{N}x, 1\}), 2^{N+1}x \} ,1\},
\]
 where $\Delta_P$ is the syntactic modulus of continuity of $P$ in $\mathcal{L}$, and $\Delta_Q(x) = x$. 


$T_\mathscr{S}$ has 
\[
  \DS x Q(x)
\]
as an axiom (i.e.\ a closed formula asserting that $Q$ is the distance predicate of a singleton). 
Let $\ast$ be a constant referring to the unique point defined by $Q$. (We add this constant in order to make the following axioms easier to write down, but it is not strictly necessary.)

Let $f:[0,1]\rightarrow[0,1]$ be a total computable continuous function whose zeroset is precisely $X=\{0\}\cup\{2^{-n}:n<\omega\}$. $T_\mathscr{S}$ has the axioms
\begin{align*}
  \forall x& f(d(x,\ast)) = 0\text{ and} \\
  \D x& |d(x,\ast) - r|\text{ for each }r\in X.
\end{align*}
The first axiom listed here should be thought of as saying `$\forall x d(x,\ast) \in X$.'


 By abuse of notation, label those definable sets $O_n$. The sentence $\Xi_\mathscr{S}$ is given by
$$ \Xi_\mathscr{S} = \sup_{x,y\in O_0} d(x,y),$$
i.e.\ the diameter of $O_0$. 
Finally, for each predicate symbol $P$ on the sort $O_{n_1}\times\dots \times O_{n_{k}}$, there are the axioms
\[
  \forall \bar{x}  \left(\bigwedge_{0<i\leq k}d(x_i, O_{n_i}) > 2^{-n_i-2} \to P(\bar{x}) = 1\right)\text{ and}
\]
\[
  \left(\forall \bar{x},\bar{y}  \in O_{n_1}\times \dots \times O_{n_k}\right)  |P(\bar{x})-P(\bar{y})| \leq \Delta_P(\min\{1,\max_{0<i\leq k}2^{n_i} d(x_i,y_i)\}),
\]
i.e.\ $P$ is $1$ outside of its original domain and inside its original domain it obeys the modulus of uniform continuity that it originally obeyed after distances are scaled correctly.
  \end{proof}

The following proposition is clear by construction and in particular by part \emph{(iii)} of Proposition \ref{prop:cmdu} above.
\begin{prop} \label{prop:one-sort}
If $\mathcal{L}$ is a metric signature with countably many sorts
and we let $U=\bigsqcup_{O\in\mathcal{S}}^{\ast}O$ be the imaginary
disjoint union of all $\mathcal{L}$-sorts, then for all $\mathfrak{A}\in C_{\mathcal{L},r}$, we have that
$\mathfrak{A}$ and $U^{\mathfrak{A}}$ have uniformly definable imaginary
expansions which are uniformly interdefinable.

\end{prop}

Aside from the issue of topological dimension and continuous degrees of points in the structure discussed in the introduction, one of the mild technical advantages of a countable metric disjoint union over
an $\omega$-product is that parameters in non-trivial
$\omega$-products tend to be poorly behaved in that they act like
 countable collections of parameters rather than finite collections
of parameters. This general phenomenon of single parameters acting
like countable collections of parameters can be blamed for many of the pathologies 
in continuous logic (e.g.\ pairs $ab$ such that $\text{tp}(ab)$ is principal but $\text{tp}(a/b)$ is not, theories with exactly two separable models, small theories with only `approximately $\omega$-saturated' separable models, and $\omega$-categorical theories which fail to be $\omega$-categorical after naming an element). In \cite{Yaacov2007}, Usvyatsov and Ben Yaacov introduced the notion of a $d$-finite type, which, intuitively speaking, characterizes when a finitary type actually behaves like a discrete finitary type rather than a discrete $\omega$-type. Uniform $d$-finiteness is a technical strengthening of $d$-finiteness that was needed in an analog of Lachlan’s theorem on the number of countable models of a superstable theory in \cite{Yaacov2007}.

\begin{prop}\phantomsection\label{prop:not-bad}
  \begin{enumerate}[label=(\roman*)]
  \item Let $\overline{a}\in\bigsqcup_{n<\omega}^{\ast}O_{n}$ be an $\ell$-tuple
of elements not equal to $\ast$. For any set $B$ of parameters,
$\text{tp}(\overline{a}/B)$ is (uniformly) $d$-finite as a type in the correct
product sort if and only if it is (uniformly) $d$-finite as a type in the sort
$\left(\bigsqcup_{n<\omega}^{\ast}O_{n}\right)^{\ell}$. (Note that since
$\ast\in\text{dcl}(\varnothing)$, its type is always uniformly $d$-finite
and adding it to a tuple preserves $d$-finiteness.)
\item For any (locally) compact set $B\subset O_{k}^{\mathfrak{A}}$,
the corresponding set in $\bigsqcup_{n<\omega}^{\ast}O_{n}^{\mathfrak{A}}$
is (locally) compact. (Although note that the countable metric disjoint union will typically fail to be locally compact at $\ast$.)
\item For any topologically finite dimensional (resp.\ weakly infinite dimensional) set  $B\subset O_{k}^{\mathfrak{A}}$,
the corresponding set in $\bigsqcup_{n<\omega}^{\ast}O_{n}^{\mathfrak{A}}$ is finite dimensional  (resp.\ weakly infinite dimensional). If each $O^{\mathfrak{A}}_k$ is finite dimensional, then $\bigsqcup_{n<\omega}^{\ast}O_{n}^{\mathfrak{A}}$ will be either finite dimensional or weakly infinite dimensional and locally finite dimensional away from $\ast$. If each $O^{\mathfrak{A}}_k$ is weakly infinite dimensional, then $\bigsqcup_{n<\omega}^{\ast}O_{n}^{\mathfrak{A}}$ is as well.
  \end{enumerate}
\end{prop}

\begin{proof}
These all follow from the fact that the natural inclusion maps $O_{k}\rightarrow\bigsqcup_{n<\omega}^{\ast}O_{n}$
are open, isometric-up-to-scaling, and bijections between definable sets.
\end{proof}
In particular if $T$ is `hereditarily $\omega$-categorical' (i.e.\ $\omega$-categorical over every finite set of parameters) or has an exactly $\omega$-saturated separable model, then $\text{Th}\left(\bigsqcup_{n<\omega}^{\ast}O_{n}\right)$ will as well \cite{Yaacov2007}.

\section{\label{sec:Making-everything-Lipschitz}Making Everything Lipschitz}

\noindent Ultimately we will need all of our predicate symbols to be Lipschitz since they will be encoded directly into a metric and metrics are always Lipschitz. There are a couple of ways to accomplish this. If the reader does not care about computability, this
section can be skipped using the following Fact~\ref{fact:lip}. 
 Also it
should be noted that Fact \ref{fact:lip} does not rely on the signature in question being countable,
but the result that we will use, Proposition \ref{prop:everything-lip}, does in general.
\begin{fact} \label{fact:lip}
Let $(X,d)$ be a metric space and $f:X\rightarrow[0,1]$ be a uniformly
continuous function. For each $0<n<\omega$, let

\[
f_{n}(x)=\inf_{y}\min\left\{ \frac{1}{n}f(y)+d(x,y),1\right\}.
\]

Then $f_{n}(x)$ is a sequence of $1$-Lipschitz functions such that
$nf_{n}\rightarrow f$ uniformly as $n\rightarrow\infty$.
\end{fact}

In general the transformation in Fact~\ref{fact:lip} would cost a jump to compute on a given structure, i.e.\ if some degree $\mathbf{a}$ computes a structure $(M,P)$ with predicate $P$, then $\mathbf{a}^\prime$ will compute $(M,P_0,P_1,\dots)$ with $P_n$ given by the formula in Fact \ref{fact:lip}.  So to ensure that the construction is computable, we will have to use something else. We will use the fact that if $\alpha$ is a concave non-decreasing function such that $\alpha(0)=0$, then for any metric $d$, $\alpha(d)$ is also a metric. If one of our predicates $P$ has a concave non-decreasing modulus of uniform continuity, then this means that we can compose it with the metric to get a uniformly equivalent metric relative to which $P$ is $1$-Lipschitz.

The following is a fairly elementary real analytic fact, but we will
include a proof for the sake of demonstrating that the procedure is computable. Note that we could avoid this lemma entirely if our moduli of uniform continuity were non-decreasing and sub-additive, which is often required and can always be arranged as shown in this lemma.
\begin{lem} \label{lem:concave}
Let $\delta:[0,1]\rightarrow[0,1]$ be a continuous function satisfying
$\delta(0)=0$. There is a continuous, concave, non-decreasing
function $\alpha:[0,1]\rightarrow[0,1]$ satisfying $\alpha(0)=0$
and $\alpha\geq\delta$. Furthermore, $\alpha$ is uniformly computable from $\delta$. 
\end{lem}

\begin{proof} $\alpha$ will be the `non-decreasing convex hull of $\delta$,' defined
by the following formula:

\[
\alpha(x)=\inf\{mx+b\,:\, 0 \leq m,b,\,(\forall y\in[0,1])my+b\geq\delta(y)\}.
\]






For each $n<\omega$, define
$$\alpha_n (x)=\inf\{mx+b\,:\, 0 \leq m,b,\,(\forall k\in \{0, 1, \dots, 2^n \})m(2^{-n}k)+b\geq\delta(2^{-n}k)\}.$$
When computing $\alpha_n$, the largest $m$ necessary is at most $$m_n= 2^n \sup_{0\leq k < 2^n}  |\delta(2^{-n}(k+1))-\delta(2^{-n} k)|,$$
and the largest $b$ is always at most $1$, so the computation of $\alpha_n$ amounts to 
 minimizing a $\delta$-computable linear function on a $\delta$-computable bounded polytope, so the $\alpha_n$ are uniformly computable in $\delta$ \cite[Chapter 5]{Weihrauch:2000:CAI:358357}. 
Furthermore note that since each $\alpha_n$ is the infimum of a family of Lipschitz functions with uniformly bounded Lipschitz coefficients, $\alpha_n$ is Lipschitz and in particular continuous.

Now all we need to show is that $\alpha_n$ converges uniformly to $\alpha$ with a computable modulus of uniform convergence. For computability considerations, we will need the fact that the modulus of uniform continuity of a continuous function $f$ on $[0,1]$ is uniformly computable from $f$ \cite[Chapter 6]{Weihrauch:2000:CAI:358357}. Let $\Delta_\delta$ be the modulus of uniform continuity of $\delta$. 
 By replacing $\Delta_\delta$ with $\sup_{0\leq y \leq x} \Delta_\delta (y)$ (which is uniformly computable from $\Delta_\delta$, since $[0,x]$ is effectively compact uniformly in $x$), we may assume that $\Delta_\delta$ is non-decreasing.

Now note that for each $n<\omega$, we have the following inequality:
\begin{equation}
\alpha_n \leq \alpha \leq \alpha_n + 2\Delta_\delta(2^{-n}). \tag{$\star$}
\end{equation}
To see that this inequality is true, observe that for each interval $I=\left[2^{-n}k,\allowbreak 2^{-n}(k+1)\right]$, we must have
\begin{align*}
  \delta(x)&\leq  \max\left\{   \delta(2^{-n}k), \delta(2^{-n}(k+1))\right\}  + \Delta(2^{-n}) \\
  &\leq  \min\left\{ \delta(2^{-n}k) , \delta(2^{-n}(k+1)) \right\}  + 2 \Delta(2^{-n}),
\end{align*}
for all $x\in I$. Therefore, if $m,b\geq 0$  satisfy the requirements in the infimum defining $\alpha_n$, then for all $x\in I$, 
$$mx+b\geq\max\{ \delta(2^{-n}k), \delta(2^{-n}(k+1))\},$$
and thus $(\star)$ follows, so we get that $\alpha_n\rightarrow \alpha$ uniformly as $n\rightarrow \infty$, and $\alpha$ is continuous. Furthermore, we clearly have a uniformly computable modulus of uniform convergence, so $\alpha$ is uniformly computable.

Finally note that $\alpha$ is concave and non-decreasing by construction (these are preserved by infima) and $\alpha(0)=0$ since for every $\varepsilon > 0$, there is an $m>0$ such that $mx+\varepsilon \geq \delta(x)$ for all $x\in[0,1]$ by continuity of $\delta$.
\end{proof}

So as long as we have a single modulus of continuity that all relation symbols obey, we can find an inter-definable structure with a Lipschitz signature. We can always arrange this if our signature is countable.

\begin{defn}[Uniform uniform continuity]
  \begin{enumerate}[label=(\roman*)]
  \item A family of functions $f\in F$ on a metric space $X$ is \emph{uniformly uniformly continuous} or \emph{u.u.c.}\ if there is a single modulus of uniform continuity valid for all $f\in F$.
  \item A metric signature $\mathcal{L}$ is \emph{u.u.c.}\ if $\Delta_{P}=\Delta_{Q}$ for all predicate symbols $P$
and~$Q$.
  \end{enumerate}
\end{defn}

Recall that two metric spaces $(X_0,d_0)$ and $(X_1,d_1)$ are bi-uniformly isomorphic if there is a uniformly continuous bijection $f:X_0 \rightarrow X_1$ with uniformly continuous inverse. Two metrics $d_0,d_1$ on the same space $X$ are uniformly equivalent if $(X,d_0)$ and $(X,d_1)$ are bi-uniformly isomorphic under the identity map.

\begin{lem}
  \begin{enumerate}[label=(\roman*)]
  \item If $d$ is a $[0,1]$-valued metric and $\alpha:[0,1]\rightarrow[0,1]$
is a continuous, concave, non-decreasing function satisfying
$\alpha(0)=0$, then $\max\{\alpha(d), d\}$ is a metric
that is uniformly equivalent to $d$.
\item If $(X,d)$ is a metric space with diameter $\leq 1$ and $f_{i}:X\rightarrow[0,1]$ for $i\in I$
is a family of u.u.c.\ functions with continuous,
sub-additive, non-decreasing modulus of uniform continuity $\alpha$,
then $(X,\max\{\alpha(d), d\})$ is a metric space bi-uniformly isomorphic
to $(X,d)$, such that the family $\{f_{i}\}_{i\in I}$ is $1$-Lipschitz.
  \end{enumerate}
\end{lem}

\begin{proof}
(i) Concave functions are sub-additive. The pseudo-metric axioms are preserved under
composition with sub-additive, non-decreasing functions which fix $0$, so $\alpha(d)$ is a pseudo-metric. The maximum of two pseudo-metrics is still a pseudo-metric, so $\max\{\alpha(d),d\}$ is a pseudo-metric. $\max\{\alpha(d), d\}=0$
if and only if $d=0$, so it is actually a metric. $\max\{\alpha(d), d\}$ and $d$ are clearly uniformly equivalent. 

(ii) This is immediate from (i).
\end{proof}

In the previous lemma we only need to take the maximum with $d$ on the off chance that $\alpha=0$. Ultimately there is no harm in doing so.

\begin{lem}
 If $\mathcal{L}$
is a countable metric signature, then it is interdefinable with a
u.u.c.\ metric signature $\mathcal{K}$. Furthermore if $\mathcal{L}$ is computable, then we can
take $\mathcal{K}$ to be uniformly computable in $\mathcal{L}$.
\end{lem}

\begin{proof}
Let $\{P_{i}\}_{i<\omega}=\mathcal{P}$ be an enumeration of all the
predicate symbols in $\mathcal{L}$ (in any sort). For each $i<\omega$,
let $Q_{i}$ be the $\mathcal{L}$-formula $2^{-(i+1)}P_{i}$. The
$\mathcal{L}$-formulas $Q_{i}$ are u.u.c.\ with regards to the modulus of uniform continuity $\Delta = \sum_{i<\omega}2^{-(i+1)}\Delta_{P_{i}}$.
If we let $\mathcal{K}$ be a metric signature with the same sorts
as $\mathcal{L}$ and predicate symbols for the $Q_{i}$, each with $\Delta_{Q_i} = \Delta$, then $\mathcal{K}$
is the required metric signature.

The procedure described in Lemma \ref{lem:concave} is uniformly computable, so passing from $\mathcal{L}$ to $\mathcal{K}$ is uniformly computable as well.
\end{proof}

\begin{prop}\phantomsection\label{prop:everything-lip}
  \begin{enumerate}[label=(\roman*)]
  \item If $\mathcal{L}$ is a countable metric signature, then it is interdefinable
with a $1$-Lipschitz metric signature $\mathcal{K}$, i.e.\ a signature
such that $\Delta_{P}(x)=x$ for all predicate symbols $P$ (although
not for metrics, which are necessarily $2$-Lipschitz). Furthermore $\mathcal{K}$ is uniformly computable from $\mathcal{L}$.
\item There is a $\mathcal{K}$-theory $T_\mathcal{L}$ such that the models of $T_\mathcal{L}$ are precisely the interpretations of $\mathcal{L}$-structures as $\mathcal{K}$-structures. Furthermore $T_\mathcal{L}$ is uniformly computable from $\mathcal{L}$.
  \end{enumerate}
\end{prop}

\begin{proof}
(i) Aside from what we have already outlined in this section, the only subtlety is that the passage from $d$ to $\max\{\alpha(d), d\}$ may delete some information contained in $d$ because of `clipping' wherever $\alpha$ is locally constant (and therefore not locally invertible). 
 To remedy this all we need to do is add, for each sort $O$, a new binary $1$-Lipschitz predicate symbol $P_{d,O}$  whose interpretation is $\frac{1}{2}d_O$ before running the construction in this section. This does not prevent $\alpha$ from clipping the metric, but we lose no information since we can recover the original metric from this predicate. 

 (ii) $T_\mathcal{L}$ just needs to express that every predicate symbol is uniformly continuous with regards to the original metrics $d_O = 2 P_{d,O}$ in the appropriate way, i.e.\ with axioms of the form
%
%
$$ \forall xy |P(x)-P(y)| \leq \Delta_P(2 P_{d,O}(x,y)) $$
and analogous axioms for predicates on more than one sort.
\end{proof}

\section{Encoding in Metric Spaces}


\noindent Most of the coding tricks used in the two following constructions boil down to the fact that if $X$ and $Y$ are metric spaces with diameter $\leq 1$, then for any $1$-Lipschitz function $f:X\times Y \rightarrow [0,1]$, you can extend the metrics on $X$ and $Y$ to $X \sqcup Y$ with $d(x,y)= 2 + f(x,y)$ for $x\in X$ and $y\in Y$. After doing this, if $X$ and $Y$ are definable from the metric, we can recover $f$ from the metric alone. The other fundamentally important thing is that since our metric structures have bounded diameter, we can add points at a larger diameter to ensure that they are $\varnothing$-definable in terms of the metric regardless of the content of the embedded metric structure.

For the sake of simplicity and to avoid writing a large number of fractions, we will write metrics with distances that are larger than 1. To bring this into line with the $[0,1]$-valued metric convention established at the beginning of the paper, divide all distances by $6$.

\begin{thm}\phantomsection\label{thm:main2}
  \begin{enumerate}[label=(\roman*)]
  \item If $\mathcal{L}$ is a countable metric signature, then for any $C_{\mathcal{L},r}$, there is a uniformly
definable imaginary $X$ such that for any $\mathfrak{A}\in C_{\mathcal{L},r}$,
$\mathfrak{A}$ and the purely metric reduct $X_{0}^{\mathfrak{A}}=(X^{\mathfrak{A}},d)$
are uniformly bi-interpretable in the sense that
\begin{itemize}
\item there are uniformly definable imaginary expansions of $\mathfrak{A}$
and $X_{0}^{\mathfrak{A}}$ which are uniformly interdefinable, and
\item there are uniformly definable bijections between the sorts of $\mathfrak{A}$
and definable subsets of $X_{0}^{\mathfrak{A}}$, and $X_{0}^{\mathfrak{A}}$
is contained in the definable closure of the images of those bijections.
\end{itemize}

Furthermore the interpretation preserves embeddings and (uniform) $d$-finiteness of types. If the original structure is not strongly infinite dimensional, then the interpreted structure will also not be strongly infinite dimensional. The interpretation preserves local compactness and local finite dimensionality away from a fixed compact $\varnothing$-definable set of bad points.
\item For any countable metric signature
$\mathcal{L}$, there is a first-order theory $T_{\mathcal{L}}$ and
a sentence $\Xi$ such that for any $r\in(0,1]$, the class of metric
spaces of the form $X_{0}^{\mathfrak{A}}$ for $\mathfrak{A}\in C_{\mathcal{L},r}$
is precisely the set of models of $T_{\mathcal{L}}\cup\{ \Xi \geq r\}$.
If $\mathcal{L}$ is a computable signature, then $T_\mathcal{L}$ is computable.
$\Xi$ does not depend on $\mathcal{L}$ and is always computable.

Furthermore there are computable mappings of presentations of $\mathcal{L}$-structures to presentations of models of $T_\mathcal{L}$ and vice versa (these mappings do not depend on $r$). 
  \end{enumerate}
\end{thm}

\begin{proof}
(i) By applying Lemma \ref{lem:prod}, we may assume that $\mathcal{L}$ has a uniform arity bound of 2. By applying Propositions \ref{prop:one-sort} and \ref{prop:everything-lip}, we may assume that $\mathcal{L}$ has a single sort and is $1$-Lipschitz. By recasting unary predicates $P$ as binary predicates using $P(x,y)=P(x)$, we may assume that all predicates are binary.

Let $\{P_n\}_{n<\omega}$ be an enumeration of all predicates with $P_0(x,y) = \frac{1}{2}d(x,y)$.

$X^\mathfrak{A}$ will have the set $A \sqcup A\times \omega \sqcup\{\infty,t\}$ as its universe, where $A\times \omega \sqcup \{\infty\} $ will be a modified countable metric disjoint union, with overflow point $\infty$, and $t$ will be a tag to keep things straight. $X^\mathfrak{A}$ will have the unique metric defined by

\begin{itemize}
\item $d(x,y)=d^\mathfrak{A}(x,y)$ for $x,y\in A$,
\item $d(x,(y,n))=2+2^{-n-1}d^\mathfrak{A}(x,y)$ for $x\in A$ and $(y,n)\in A \times \omega$,
\item $d(x,\infty)=2$ for $x \in A$,
\item $d(x,t)=5$ for $x\in A$,

\item $d((x,n),(y,n))=2^{-n}d^\mathfrak{A}(x,y)$ for $(x,n),(y,n)\in A \times \omega$,
\item $d((x,n),(y,n+1))=2^{-n-1}(1+P_n^\mathfrak{A}(x,y))$ for $(x,n),(y,n+1)\in A \times \omega$,
\item $d((x,n),(y,m))=|2^{-n}-2^{-m}|$ for $(x,n),(y,m)\in A \times \omega$ with $|n-m|>1$,

\item $d((x,n),\infty)=2^{-n}$,
\item $d((x,n),t)=4+2^{-n-1}$ for $(x,n)\in A \times \omega$, and
\item $d(\infty,t)=4$.

\end{itemize}
All of the metric space axioms except for the triangle inequality are clearly obeyed by $d$. If all three points are in the same copy of $A$ then the triangle inequality is obeyed, so we only need to check mixed triples. The majority of cases are mechanical to check, but there are a handful of tight or subtle cases that we will write out explicitly. Let $x,y\in A$ and $(z,n),(w,n),(u,n+1),(v,n+1),(s,m)\in  A\times \omega$ with $|n-m|>1$, where $n,m <  \omega$. Also, recall that if $n\neq m$, then $|2^{-n}-2^{-m}|\geq 2^{-n-1}$. Here are the cases we check explicitly:
\begin{itemize}
\item $d(x,(z,n)) \leq  2 + 2\cdot2^{-n-2}\leq d(x,(u,n+1)) + d((u,n+1),(z,n))$
\item $d((z,n),(u,n+1))\leq 2^{-n-1}(1+P_n(w,u)+d(z,w))\\ \leq d((z,n),(w,n))+d((w,n),(u,n+1))$
\item $d((z,n),(w,n)) \leq 2\cdot2^{-n-2} \leq d((z,n),(u,n+1)) + d((u,n+1),(w,n))$
\item $d((z,n),(v,n+1)) \leq 2^{-n-1}(1+P_n(z,u)+d(u,v)) \\ = d((z,n),(u,n+1))+d((u,n+1),(v,n+1))$
\item $d((z,n),t) = 4 + 2 \cdot 2^{-n-2} \leq d((z,n),(u,n+1))+d((u,n+1),t)$
\end{itemize}
 Just as in the proof of Proposition \ref{prop:cmdu}, let $Y={0}\cup\{2^{-n}:n<\omega\}$ and let $Q:Y\rightarrow [0,1]$ be the natural inclusion map, which is a definable predicate on $Y$. For each $n$, let
$$ R_n(x)= 1 \dotdiv 2^{n+1}|Q(x)-2^{-n}|,$$
and define a pseudo-metric, $\rho$, on $Y\times A$ by
\begin{align*}
\beta_n(\overline{x},\overline{y}) &= R_n(x_0)R_{n+1}(y_0) P_n (x_{n+1},y_{n+2}) + R_{n+1}(x_0) R_n(y_0) P_n(y_{n+1},x_{n+2})\text{ and} \\
  \rho(\overline{x},\overline{y}) &=|Q(x_0)-Q(y_0)| +  \sum_{n<\omega} 2^{-n} \left(R_n(x_0)R_n(y_0) d(x_{n+1},y_{n+1}) +\frac{1}{2} \beta_n(\overline{x},\overline{y}) \right).
\end{align*}
Then $Y\times A/\rho$ will correspond to $A\times \omega \sqcup \{\infty\} $, where $\infty$ is the $\rho$-equivalence class of any element of the form $\left<0,x\right>$ for $x\in A$.

Recall that an element or set is definable if there is a formula which defines its distance predicate. If we have a $\{0,1\}$-valued indicator function, $\varphi(x)$, for the set $\varphi^{-1}(0)$, then that is even better and we can always define the distance to the set by $d(x,\varphi^{-1}(0))=\inf_y d(x,y) + 5\varphi(y)$ if we need it. Once a point is definable, we will freely use it as a constant to make the following formulas simpler \cite[Proposition 9.18]{MTFMS}. 

First note that the formulas 
\begin{align*}
  C(x) &= \forall y d(x,y) = 0 \vee d(x,y) \geq 4\text{ and} \\
  O(x) & = \sforall y d(x,y) < 1 \vee d(x,y) > 3
\end{align*}
are both satisfied if and only if $x=t$ because $t$ is the only point for which there is no $y$ with $1\leq d(t,y) \leq 3$. This implies that $\mathrm{tp}(t)$ is topologically isolated and so $\{t\}$ is a definable singleton. Therefore we can use it as a constant to define distance predicates for $A$ and each $A \times \{n\}$:
\begin{align*}
 d(x,A)&= \inf_y d(x,y) + 2 |d(y,t) - 5|\text{ and} \tag{$\dagger$}\\
d(x,A\times\{n\}) &= \inf_y d(x,y) + 2 |d(y,t)-(4+2^{-n-1})|. \tag{$\dagger$}
\end{align*}
These formulas are distance predicates by our choice of distances to $t$. 
$|d(y,t) - 5|$ and $|d(y,t)-(4+2^{-n-1})|$ roughly give the distances to $A$ and $A\times\{n\}$ and then the method used in the proof of Proposition 9.19 in \cite{MTFMS} gives an exact distance predicate.


For each $n<\omega$, there is a definable bijection from $A$ to $A\times\{n\}$ given by
\begin{equation}
d(y,f_n(x))= 2^{n+1}(d(x,y)\dotdiv 2), \tag{$\circ$}
\end{equation}
and so for any $n<\omega$, we can define $P_n$ on $A$ by
\[
P_n(x,y) = 2^{n+1}(d(f_n(x),f_{n+1}(y)) \dotdiv 1).
\]
So $X$ is the required uniformly definable imaginary, which clearly preserves embeddings. The interpretation preserves (uniform) $d$-finiteness of types, lack of strong infinite dimensionality, local compactness, and local finite dimensionality by the same argument as in the proof of Lemma \ref{prop:not-bad} (specifically, the inclusion maps are open isometries-up-to-scaling).

The advertised set of bad points is $\{\infty,\ast\}\cup \{(\ast,n) : n <\omega\}$. 
 Since this is a closed compact set of $\varnothing$-definable points, it is algebraic over $\varnothing$.

(ii) $T_\mathcal{L}$ is a theory in the language of metric spaces of diameter $5$. By Lemma~\ref{lem:prod} and Propositions \ref{prop:cmdu} and \ref{prop:everything-lip}, we only need to construct $T_\mathcal{L}$ in the case where $\mathcal{L}$ has one sort and is 1-Lipschitz.


$T_\mathcal{L}$ contains the axioms
\[
  \wexists x \forall y  (x=y \vee d(x,y) \geq 4)\text{ and} 
\]
\[
  \forall xy  \left( \sforall z( d(x,z) < 1 \vee d(x,z) > 3 ) \wedge \sforall z( d(x,z) < 1 \vee d(x,z) > 3 )  \to x=y \right),
\]
i.e.\ there is a unique element $x$ with the property that every distance to it is either less than $1$ or greater than $3$, and, furthermore, this element actually has the property that every distance to it is either $0$ or at least $4$.


Since there is an open formula satisfied by a single element, it is actually definable as a singleton. Let $t$ denote that element for the sake of making the following axioms simpler to write down. Let $f:[0,5]\rightarrow[0,1]$ be a computable total continuous function whose zeroset is precisely $Z = \{4 + 2^{-n-1} : n<\omega\} \cup \{4,5 \} $. $T_\mathcal{L}$ has the axioms
\begin{align*}
  \forall x & f(d(x,t)) = 0\text{ and} \\
  \wexists x & d(x,t) = r\text{ for all }r \in Z,
\end{align*}
i.e.\ distances to $t$ are always in $Z$, and every distance in $Z$ is attained in some elementary extension. For isolated points in $Z$ (everything other than $4$), this implies that the distances are attained in every model.

$T_\mathcal{L}$ also has axioms
\begin{align*}
  \D x & ( \inf_y d(x,y) + 2|d(y,t) - 5| )\text{ and}\\
  \D x & ( \inf_y d(x,y) + 2|d(y,t) - (4+2^{-n-1})|  )\text{ for each }n<\omega,
\end{align*}
which assert that the formulas $(\dagger)$ are distance predicates of definable sets (specifically, $A$ and the $A \times \{n\}$). We will now write these formulas as $d(x,A)$ and $d(x,A_n)$, respectively, and refer to the corresponding definable sets as $A$ and $A_n$.
We also need to actually assert that these functions are isometries-up-to-scaling, which can be done with
\begin{align*}
  (\forall x_0,x_1 \in A) (\forall y_0,y_1 \in A_n) & |d(x_0,x_1) - 2^n d(y_0,y_1)| \\ & \leq 2^{n+3} [ (d(x_0,y_0)\dotdiv 2)\allowbreak + (d(x_0,y_1)\dotdiv 2) ] .
\end{align*}


We need axioms enforcing the definition of $d$ given in part (i) of this proof other than the line involving $P_n$ (which isn't determined by $T_\mathcal{L}$) and lines involving $\infty$ (which are automatically enforced by continuity). The distances between $A$ and $A_n$ are already enforced by the previous axioms. We need
\begin{align*}
(\forall x \in A) & d(x,t) = 5, \\ 
(\forall x \in A_n) (\forall y \in A_m) & d(x,y) = |2^{-n} - 2^{-m}|,\text{ and} \\  
(\forall x \in A_n) & d(x, t) = 4+2^{-n-1}. 
\end{align*}
For the $P_n$ line we just need to enforce the lower bound of $2^{-n-1}$ and the upper bound of $2^{-n}$ and to ensure that $P_n$ (which is definable from $d$ since we can define the sets $A_n$) obeys the correct modulus of uniform continuity (relative to the predicate $2 P_0$, since the metric itself may have lost information to clipping). This is accomplished by
\begin{align*}
  (\forall x \in A_n)&(\forall y \in A_{n+1})  2^{-n-1}\leq d(x,y) \leq 2^{-n}\text{ and} \\
  (\forall x_0,x_1,y_0,y_1 \in A)& |P_n(x_0,x_1) - P_n(y_0,y_1)| \leq \Delta_{P_n}(2\max\{P_0(x_0,y_0), P_0(x_1,y_1)\}).
\end{align*}

 For those predicate symbols that were originally unary we need axioms enforcing that $P_n$ only depends on one input, namely 
 \[
(\forall xy_0y_1 \in A) P_n(x,y_0)=P_n(x,y_1)
 \]
 for each unary $P_n$.

 The existence of $\infty$ and its definability are implied by these other axioms (since the $A$ form a Cauchy sequence of definable sets in the Hausdorff metric whose diameters are limiting to $0$ and a Hausdorff metric limit of definable sets is definable). Finally $\Xi$ is given by
 \[
  \Xi = \sup\limits_{x,y \in A} d(x,y), 
 \]
 which evaluates to the diameter of the set $A$.
\end{proof}

Assuming that the signature has finitely many sorts and a uniform arity bound (but maybe infinitely many predicate symbols) we can avoid the bad points entirely, but the construction is different. It is somewhat less delicate than the construction in Theorem \ref{thm:main2}, so  we'll only sketch the important specifics. 

\begin{thm} \label{thm:main3}
If $\mathcal{L}$ is a countable metric signature with finitely many sorts and a uniform arity bound, then the result of Theorem \ref{thm:main2} holds with no bad points, i.e.\ the bi-interpretation preserves local compactness and finite dimensionality everywhere.
\end{thm}

\begin{proof}
By applying Proposition \ref{prop:everything-lip}, we may assume that $\mathcal{L}$ is $1$-Lipschitz. Let $\{O_n\}_{n<k}$ be a finite list of all base sorts and let $\{N_n\}_{n<\ell}$ be a finite list of all finitary product sorts of the form $\prod_{O\in a(P)}O$ for some predicate symbol $P$. The sort $X$ will be constructed from a graph with the following nodes:
\begin{itemize}
\item For each $n<k$, a main copy of the sort $O_n$.
\item For each $n<\ell$, a copy of $N_n=\prod_{O\in a(P)}O$ along with copies of each $O$ in $a(P)$ (with multiplicity).
\item For each $n<\ell$, a copy of $I=\{0\}\cup\{2^{-s} : s<\omega\}$.

\end{itemize}

Connections between the nodes will correspond to specific relationships being encoded in the metric.
\begin{itemize}
\item For each main copy of $O_n$ and each copy of $O_n$ associated to some $N_m$ there is an edge. Call the associated copy $O^\prime_n$. The metric between $x\in O_n$ and $y \in O^\prime_n$ will be given by $d(x,y)=2+d_{O_n}(x,y)$, in order to encode a definable bijection between $O_n$ and $O^\prime_n$.
\item For each $N_m$ and associated $O^\prime_n$ there is an edge. If $O^\prime_n$ is the $i$th factor of $N_m$, then the metric between $\overline{x} \in N_m$ and $y\in O^\prime_n$ will be given by $d(\overline{x},y)=2+d_{O_n}(x_i,y)$, in order to encode a definable projection from $N_m$ to $O_n^\prime$.
\item For each $N_m$ and its associated copy $I_m$ of $I$ there is an edge. Let $\{P_n\}_{n<\omega}$ be a list of the predicates symbols on $N_m$. If $\overline{x}\in N_m$ and $2^{-n}\in I$, then $d(\overline{x},2^{-n})=2+2^{-n}P_n(\overline{x})$ and $d(\overline{x},0)=2$. (This is where it's important that the predicate symbols be $1$-Lipschitz. If $P_n$ is not $1$-Lipschitz, this formula cannot define a metric).
\end{itemize}

Let all other distances be 4. Finally add a single new point $t$, with distances to everything else between $5$ and $6$ chosen to make each node of the graph have a $t$-definable indicator function. Then using the same kind of formula as in the proof of Theorem \ref{thm:main2}, $t$ is $\varnothing$-definable, so each of the nodes in the graph is definable as well.

Note that for any function $\eta(x)$ taking on $0$ on some copy of $I$ and $1$ everywhere else, the formula
$$ \min\left\{   \eta(x) + 8 \sup_y \min\left\{   d(x,y) , \frac{1}{2} \dotdiv d(x,y)  \right\}, 1\right\} $$
is $\{0,1\}$-valued and takes on the value $0$ if and only if $x$ is $1$ (as an element of $I$). Therefore we can use $1\in I$ as a constant, and for each $n\leq \omega$, we can define a distance predicate for $2^{-n}\in I$ (with $2^{-\omega}=0$) by
$$d(x, 2^{-n}) = \inf_y d(x,y)+2|d(y,1)-(1-2^{-n})|.$$ 
So each point in each copy of $I$ is $\varnothing$-definable.

Every point in $X^\mathfrak{A}$ is either an image of some Cartesian product of sorts in $\mathfrak{A}$ or contained in a compact clopen definable set (either a copy of $I$ or $t$). Finitary products preserve local compactness and local finite dimensionality, so in this construction there are no `bad points.' 
\end{proof}




\subsection['Finite Axiomatizability' in Continuous Logic]{`Finite Axiomatizability' in Continuous Logic} \label{sec:fin}

The notion of finite
axiomatizability is somewhat awkward in continuous logic. There are several possible definitions that suggest themselves, but none of them seem useful. This is the most literal transcription of the ordinary definition:

\begin{defn}[Finite axiomatizability version 1]
A theory $T$ is \emph{finitely axiomatizable} if and only if it is axiomatized by a finite collection of sentences.
\end{defn}

Depending on what we mean by `sentence,' every theory in a countable language is finitely axiomatizable in that continuous logic naturally has an infinitary conjunction of the form $\Sigma_{n<\omega}2^{-n} \varphi_{n}$, and we can just let $\varphi_n$ be an enumeration of a countable dense subset of the logical consequences of $T$.

A sensible attempt to avoid this would be a definition like this:

\begin{defn}[Finite axiomatizability version 2]
A theory $T$ is \emph{finitely axiomatizable} if and only if it is axiomatized by a finite collection of restricted sentences.
\end{defn}

But this is  arbitrary and fails to have any obvious meaningful semantic consequences.

We can try a more directly semantic definition like this:

\begin{defn}[Finite axiomatizability version 3]
A theory $T$ is \emph{finitely axiomatizable} if and only if the class of models of $T$ and its complement are both elementary.
\end{defn}

This amounts to saying $[T] = \{T^\prime \in S_0(\varnothing):T^\prime \vdash T,T^\prime \text{ a complete theory}\}$ is a clopen subset of $S_0(\varnothing)$. The problem is that for any reasonable\footnote{If all function symbols in $\mathcal{L}$ have concave moduli of continuity, then $S_0(\varnothing)$ can be continuously retracted to a point by scaling all non-metric relations to 0 and then scaling the metric to 0. On the other hand, if $\Delta_f(x)=x^2$ and the metric has diameter $\leq 1$, then the sentence $$\inf_x d(x,f(x))$$ can only take on the values 0 or 1. Either there exists some $x$ such that $d(x,f(x)) < 1$, in which case $x,f(x),f(f(x)),\dots$ converges to a fixed point of $f$, or for every $x$, $d(x,f(x))=1$.} metric signature, $S_0(\varnothing)$ is connected, so the only finitely axiomatizable theories are the trivial theory and the inconsistent one. That said, `finite axiomatizability version 3' relative to a theory can be non-trivial.

At this point we could argue that clopenness in type space is too strong of a condition in continuous logic. Definable sets do not correspond to clopen subsets of type space, but rather have a more subtle topometric characterization in terms of the $d$-metric: A closed set $D\subseteq S_n(T)$ is definable if and only if $D\subseteq \text{int}\{p\in S_n(T):d(p,D)<\varepsilon\}$ for every $\varepsilon > 0$, where $\text{int} X$ is the topological interior of $X$. By analogy we could try a similar weakening of clopen as a basis for our definition of `finite axiomatizability,' but the $d$-metric relies on $T$ being a complete theory and for a complete theory $S_0(T)$ is trivial.

There are, however, contexts in which there is a meaningful non-trivial metric on $S_0(T)$ for an incomplete theory $T$. Specifically if we're examining a notion of approximate isomorphism (such as the perturbations in \cite{OnPert} or Gromov-Hausdorff distance), we get a metric on completions of $T$---
$$\rho(T_0, T_1) = \inf \{ \varepsilon : \mathfrak{A}\models T_0, \mathfrak{B}\models T_1,\mathfrak{A},\mathfrak{B}\,\text{`}\varepsilon\text{-isomorphic'} \}$$
---whatever `$\varepsilon$-isomorphic' might mean. And in this case we get a weaker notion of finite axiomatizability:

\begin{defn}[Finite axiomatizability version 4]
A theory $T$ is \emph{finitely axiomatizable relative to $\rho$}  if there is a sentence $\chi$ such that $T \vdash \chi$ and for all complete theories $T^\prime$, $T^\prime \vdash \chi \leq \rho(T^\prime, [T])$, where $\rho(T^\prime,[T])$ is the point-set distance between $T^\prime$ and $[T]$.
\end{defn}

This definition is equivalent to the topometric condition $[T]\subseteq \text{int}\{T^\prime \in S_0(\varnothing):\rho(T^\prime,[T])<\varepsilon\}$ for every $\varepsilon > 0$. It should be noted that this is a proper generalization of version 3 in that we can take our notion of approximate isomorphism to be $\mathfrak{A}$ and $\mathfrak{B}$ are $0$-isomorphic if they are isomorphic and $1$-isomorphic if they are not.

This may be a reasonable definition in some context, although as discussed in \cite{OnPert} the metrics $\rho$ are generally much more poorly behaved than the $d$-metric. In any case it's unclear what one can do with this definition. To apply it to this paper we would need to choose a notion of approximate isomorphism before we could even ask the question of whether or not the theory $T_{\mathcal{L}}$ is `finitely axiomatizable.'


\bibliographystyle{plain}
\bibliography{../ref}

\begin{thebibliography}{1}

\bibitem{Yaacov2008-ITACFO}
Ita{\"i} Ben~Yaacov.
\newblock Continuous first order logic for unbounded metric structures.
\newblock {\em Journal of Mathematical Logic}, 8(2):197--223, 2008.

\bibitem{OnPert}
Ita{\"i} Ben~Yaacov.
\newblock On perturbations of continuous structures.
\newblock {\em Journal of Mathematical Logic}, 08(02):225--249, 2008.

\bibitem{MTFMS}
Ita{\"i} Ben~Yaacov, Alexander Berenstein, C.~Ward Henson, and Alexander
  Usvyatsov.
\newblock {\em Model theory for metric structures}, volume~2 of {\em London
  Mathematical Society Lecture Note Series}, pages 315--427.
\newblock Cambridge University Press, 2008.

\bibitem{Yaacov2007}
Ita{\"i} Ben~Yaacov and Alexander Usvyatsov.
\newblock On d-finiteness in continuous structures.
\newblock {\em Fundamenta Mathematicae}, 194(1):67--88, 0 2007.

\bibitem{10.2307/25747120}
Ita{\"i} Ben~Yaacov and Alexander Usvyatsov.
\newblock Continuous first order logic and local stability.
\newblock {\em Transactions of the American Mathematical Society},
  362(10):5213--5259, 2010.

\bibitem{hodges_1993}
Wilfrid Hodges.
\newblock {\em Model Theory}.
\newblock Encyclopedia of Mathematics and its Applications. Cambridge
  University Press, 1993.

\bibitem{2014arXiv1405.6866K}
T.~{Kihara} and A.~{Pauly}.
\newblock {Point degree spectra of represented spaces}.
\newblock {\em ArXiv e-prints}, May 2014.

\bibitem{Weihrauch:2000:CAI:358357}
Klaus Weihrauch.
\newblock {\em Computable Analysis: An Introduction}.
\newblock Springer-Verlag, Berlin, Heidelberg, 2000.

\end{thebibliography}

\end{document}